\documentclass[final,hidelinks,onefignum,onetabnum]{siamart220329}

\usepackage{amsfonts}
\usepackage{geometry}
\usepackage{todonotes}
\usepackage{enumerate}
\usepackage{tikz-cd}
\usepackage{amssymb}

\newtheorem{claim}{Claim}
\newtheorem{remark}{Remark}

\newcommand{\reals}{\mathbb{R}}

\newcommand{\comment}[1]{}

\newcommand{\R}{\mathbb{R}}

\newcommand{\DiffLip}{\Gamma}
\newcommand{\Imm}{\operatorname{Imm}}

\newcommand{\Diff}{\operatorname{Diff}}

\newcommand{\WFR}{\operatorname{WFR}}

\ifpdf
  \DeclareGraphicsExtensions{.eps,.pdf,.png,.jpg}
\else
  \DeclareGraphicsExtensions{.eps}
\fi

\headers{}{E. Hartman, M. Bauer, E. Klassen}

\title{Square Root Normal Fields for Lipschitz surfaces and the Wasserstein Fisher Rao metric
}

\author{Emmanuel Hartman\thanks{Department of Mathematics, Florida State University (\email{ehartman@fsu.edu})}\and Martin Bauer\thanks{Department of Mathematics, Florida State University and University of Vienna (\email{bauer@math.fsu.edu})}\and  Eric Klassen\thanks{Department of Mathematics, Florida State University (\email{klassen@math.fsu.edu})}}

\ifpdf
\hypersetup{
  pdftitle={},
  pdfauthor={M. Bauer, E. Hartman, E. Klassen}
}
\fi

\begin{document}
\maketitle
\begin{abstract}
The Square Root Normal Field (SRNF) framework is a method in the area of shape analysis that defines a (pseudo) distance between unparametrized  surfaces. For piecewise linear (PL) surfaces it was recently proved that the SRNF distance between unparametrized  surfaces is equivalent to the Wasserstein Fisher Rao (WFR) metric on the space of finitely supported measures on $S^2$. In the present article we extend this point of view to a much larger set of surfaces; we show that the SRNF distance on the space of Lipschitz surfaces is equivalent to the WFR distance between Borel measures on $S^2$. For the space of spherical surfaces this result directly allows us to characterize the non-injectivity and the (closure of the) image of the SRNF transform. In the last part of the paper we further generalize this result by showing that the WFR metric for general measure spaces can be interpreted as an  optimization problem over the diffeomorphism group of an independent background space. 
\end{abstract}
\tableofcontents


\section{Introduction}
The investigations of this article are motivated by applications in the area of mathematical shape analysis, which seeks to quantify differences, perform classification, and explain variability for populations of shapes~\cite{younes2010shapes,srivastava2016functional,dryden2016statistical,marron2014overview}. More specifically, the results of this article concern the Square Root Normal Field distance~\cite{jermyn2017elastic} on the space of surfaces and the Wasserstein Fisher Rao metric~\cite{chizat2018interpolating,liero2018optimal} from unbalanced optimal transport. Before we describe the contributions of the current work in more detail, we will briefly summarize some results from these two areas.

{\bf Shape analysis of surfaces:} For the purpose of this article we consider a shape to be a parametrized surface or curve in $\mathbb{R}^d$, where we identify two objects if they only differ by a translation and/or a reparametrization. In practice, it is often of interest to mod out by further shape preserving group actions, such as the groups of rotations or scalings. To keep the presentation simple, we will ignore these additional finite dimensional groups.
Consequently, the resulting shape space is an infinite dimensional, non-linear (quotient) space, which makes the application of statistical techniques to analyse these types of data a highly challenging task. A common approach to overcome these difficulties can be found in the area of geometric statistics~\cite{pennec2006intrinsic,pennec2019riemannian}, in which one develops statistical frameworks based on (Riemannian) geometry. In the context of shape analysis of surfaces or curves, a variety of different metrics have been proposed for this purpose; this includes metrics induced by (right-invariant) metrics on diffeomorphism groups~\cite{younes2010shapes,miller2002metrics} and reparametrization invariant metrics on the space of immersions~\cite{srivastava2016functional,bauer2014overview,michor2007overview}, which are directly related to the investigations of the present article as we will explain next.  

In the latter approach the calculation of the distance (similarity) between two shapes reduces to two tasks: calculating the geodesic distance on the space of immersions (parametrized surfaces or curves, resp.) and minimizing over the action of the shape preserving group actions, i.e., diffeomorphisms of the parameter space and translations in $\mathbb R^d$. In general there do not exist any explicit formulas for geodesics and thus computing solutions to the geodesic boundary value problems (and thus of the distance) is a highly non-trivial task and usually has to be solved using numerical optimization techniques, see eg.~\cite{hartman2022elastic,bauer2017numerical}. 

For specific examples of Riemannian metrics, however, simplifying transformations have been developed that allow for explicit calculations of geodesics and geodesic distance. This includes in particular the family of $G^{a,b}$-metrics on the space of curves~\cite{bauer2022elastic,needham2020simplifying,mio2007shape,younes1998computable},
a family of first order Sobolev type metrics, that are often called \emph{elastic} metrics due to their connections to linear elasticity theory; see eg.~\cite{mio2007shape,charon2022shape,bauer2022elastic}.  For the specific choice of parameters $a=1$, $b=1/2$ the corresponding transformation  is the so-called Square-Root-Velocity (SRV) transform~\cite{srivastava2010shape}, which is widely used in applications; see~\cite{srivastava2016functional} and the references therein. The advantage of this transformation is that it reduces the shape comparison problem to a single optimization over the shape preserving group actions, i.e., in the setting of the present article over reparametrizations and translations. This computational simplification has led to both the development of efficient algorithms~\cite{woien2022pde,dogan2015fast,srivastava2010shape} and to analytic results on existence of minimizers and optimal parametrizations~\cite{bruveris2016optimal,lahiri2015precise,trouve2000class}. 

The family of elastic $G^{a,b}$ metrics has a natural generalization to a four parameter  family of metrics on the space of surfaces~\cite{su2020shape}. Similarly to the case of curves, simplifying transformations have also been proposed in this more complicated situation~\cite{kurtek2010novel,kurtek2011elastic,jermyn2017elastic,su2020simplifying}.  Notably, as a generalization of the SRV transform, the Square Root Normal Field (SRNF) transformation~\cite{jermyn2017elastic} 
has been introduced. In contrast to the situation for curves, the corresponding Riemannian metric for this transformation is degenerate and, furthermore, it only leads to a first order approximation of the geodesic distance. Nonetheless it defines a reparametrization invariant (pseudo-) distance on the space of surfaces,  which still allows for efficient computations using several methods of approximating the optimization over the diffeomorphism group~\cite{laga2017numerical,bauer2020numerical} and has proven successful in several applications, see~\cite{kurtek2014statistical,joshi2016surface,matuk2020biomedical,laga20214d}. 
and the references therein.

{\bf Unbalanced Optimal transport:} 
The second core theme of the present article can be found in the theory of optimal transport (OT). Since Monge's formulation of OT as a non-convex optimization problem in the space of transport maps, many formulations of the problem have been proposed to give insight to the theoretical properties of the problem as well as efficient methods for computing the solution, see~\cite{villani2003topics,villani2008optimal} for a comprehensive overview on the field.

 In classical optimal transport theory one considers  normalized (probability) distributions. It is, however, important for many applications to relax this normalization assumption and compute transportation plans between arbitrary positive measures. Motivated by this observation the theory of optimal transport has been extended to measures with different masses. This field,  called unbalanced optimal transport, has seen rapid developments in the past years and several different frameworks have been proposed \cite{chizat2018interpolating,liero2016optimal,Lombardi2015,piccoli2014generalized}. Among them is the Wasserstein Fisher Rao (WFR) distance, an interpolating distance between the quadratic Wasserstein metric and the Fisher–Rao metric, that was introduced independently by~\cite{chizat2018interpolating} and~\cite{liero2018optimal}. The WFR distance has been applied to a variety of problems where it is more natural to consider optimal transport in an  unbalanced setting. These applications range from color transfer \cite{chizat2018scaling}, to earthquake epicenter location \cite{Zhou2018TheWM} and document semantic similarity metrics \cite{Wang2020RobustDD}. Because of the growing field of applications, several algorithms have been proposed to compute  the Wasserstein Fisher Rao metric. A variation on the popular Sinkhorn algorithm to solve for an entropy regularized version of the distance was proposed by \cite{chizat2018scaling} and an alternating minimization algorithm that computes an exact solution was introduced in~\cite{bauer2022SRNF}.

\subsection{Contributions of the article}
Recently a new relationship between these two areas (shape analysis and unbalanced optimal transport) has been found. Namely, in~\cite{bauer2022SRNF} it has been shown that for triangulated surfaces the calculation of the SRNF shape distance can be reduced to calculating the WFR distance between their corresponding surface area measures. The presentation in~\cite{bauer2022SRNF} was entirely focused on the discrete (PL) setting and the proof of the result essentially reduced to  algebraic considerations.  In the first part of the present article we build the analytical tools to extend  this result to the infinite dimensional setting,  which contains in particular the original setup of the SRNF distance; the space of smooth surfaces. The main result of this part of our article -- cf. Theorem~\ref{thm:maintheorem} -- shows that the SRNF shape distance between any two Lipschitz surfaces is equal to the WFR distance between their surface area measures. 

As a direct consequence of this result we are able to answer two fundamental questions regarding the SRNF transform: since the inception of the SRNF transform, it has been understood that the map is neither injective nor surjective~\cite{jermyn2017elastic}. Characterizing the image and non-injectivity have, however, remained open problems. Recently a first degeneracy result in the context of closed surfaces has been found~\cite{klassen2020closed}. Using our equivalence result we are able to obtain a characterization of the closure of the image of this transform -- cf. Theorem~\ref{thm:image} -- and a new strong degeneracy result of the corresponding distance (non-injectivity of the transform, resp.) -- cf. Theorem~\ref{cor:degeneracy}.

In the second part we further explore the equivalence result for more general unbalanced optimal transport problems. Generalizations of some of the intermediate results of the first part  allow us to offer a novel formulation of the WFR metric as a diffeomorphic optimization problem -- cf. Theorem~\ref{thm:cone_iso}. 
\textcolor{black}{This theorem is highly related to a classical  result from optimal transport theory: in \cite{OttoPic} Otto proved the existence of  Riemannian submersion from the space of diffeomorphisms equipped with the $L^2$ metric onto the space of probability densities equipped with the Wasserstein metric. This construction was recently extended to the unbalanced setting by Gallou\"{e}t and Vialard in \cite{GALLOUET20184199}, where they constructed an analogous submersion onto the space of all densities equipped with the WFR metric. The main difference of these results to  Theorem~\ref{thm:cone_iso} of the present article is the regularity of the maps (measures) under consideration: the constructions of \cite{OttoPic,GALLOUET20184199} are in the smooth category, whereas our results are formulated for $L^2$-maps (Borel measures, resp.). Further discussion of these results and how they relate to each other are given in Remark~\ref{rem:RelationToCH}.}

\textcolor{black}{Finally, we want to emphasize that the main result of the first part of the article relates the WFR on $S^2$ with a specific choice of parameter to a diffeomorphic optimization problem. Our main result of the second part, Theorem~\ref{thm:cone_iso}, extends} this relationship to the WFR with any choice of parameter defined on any connected, compact, oriented Riemannian manifold $N$. Notably, the space of diffeomorphisms we have to optimize over does not  depend on $N$, but can be chosen as the diffeomorphism group of some background manifold, that only needs to be of dimension greater than or equal to two\textcolor{black}{, which  further distinguishes Theorem~\ref{thm:cone_iso} of the setting considered in~\cite{OttoPic,GALLOUET20184199}.} 

\subsection*{Acknowledgements}
The authors thank FX Vialard and Cy Maor for useful discussions during the preparation of this manuscript. \textcolor{black}{Moreover, we thank the anonymous reviewer for pointing out the relationship of this paper with the results of \cite{OttoPic} and \cite{GALLOUET20184199}.} 
M. Bauer was supported by NSF-grants 1912037 and 1953244 and by FWF grant P 35813-N. E. Hartman was supported by NSF grant DMS-1953244.

\section{Preliminaries}
\subsection{The Wasserstein Fisher Rao Distance}
In the following, we will summarize the Kantorovich formulation of the Wasserstein Fischer Rao distance, as introduced in~\cite{chizat2018unbalanced} \textcolor{black}{ and~\cite{kondratyev2016new}} for measures on a smooth, connected, compact, oriented Riemannian manifold, $N$. Therefore we denote by $\mathcal{M}(N)$ the space of finite Borel measures on $N$. In the Kantorovich formulation of the Wasserstein-Fisher-Rao distance, we will define a functional on the space of semi-couplings. Therefore we first recall the definition of a semi-coupling:
\begin{definition}[Semi-couplings~\cite{chizat2018unbalanced}]
Given $\mu,\nu\in\mathcal{M}(N)$ the set of all \textit{semi-couplings} from $\mu$ to $\nu$ is given by
\begin{equation*}
    \Gamma(\mu,\nu)=\left\{ (\gamma_0,\gamma_1)\in \mathcal{M}(N\times N)^2| (\operatorname{Proj}_0)_\#\gamma_0=\mu,(\operatorname{Proj}_1)_\#\gamma_1=\nu\right\}.
\end{equation*}
\end{definition}  
\noindent To define the Wasserstein-Fisher-Rao distance from $\mu$ to $\nu$ we define a functional on the space of semi-couplings from $\mu$ to $\nu$.
Let $d$ denote the geodesic distance on $N$ and $\delta\in (0,\infty)$. We consider the functional 
\begin{align*}
    J_\delta:\Gamma(\mu,\nu)&\to \R\\
    (\gamma_1,\gamma_2)&\mapsto  4\delta^2\left(\mu(N)+\nu(N)-2\int_{N\times N}\frac{\sqrt{\gamma_1\gamma_2}}{\gamma}(u,v) \overline{\cos}(d(u,v)/2\delta) d\gamma(u,v)\right)
\end{align*}
where $\gamma\in\mathcal{M}(N \times N)$ such that $ \gamma_1,\gamma_2\ll\gamma $. Note that in the case where $N=S^2$, we have $d(u,v)=\cos^{-1}(u\cdot v)$. Thus for $\delta= \frac{1}{2}$,
\begin{align}
    J_\delta(\gamma_1,\gamma_2)=\int_{S^2\times S^2}\left|\sqrt{\frac{\gamma_1}{\gamma}(u,v)}u-\sqrt{\frac{\gamma_1}{\gamma}(u,v)}v\right|^2 d\gamma(u,v).
\end{align}
\begin{definition}[Wasserstein-Fisher-Rao Distance ~\cite{chizat2018unbalanced,liero2018optimal}]
The Wasserstein-Fisher-Rao Distance on $\mathcal{M}(N)$ is given by 
\begin{align}
    \WFR_\delta:\mathcal{M}(N)\times\mathcal{M}(N)&\to\R^{\geq0} \text{ defined via }\\
    (\mu,\nu)&\mapsto \inf\limits_{(\gamma_0,\gamma_1)\in \Gamma(\mu,\nu)} \sqrt{J_\delta(\mu,\nu)}.
\end{align}
\end{definition}

\noindent Some results in this article will specifically apply to the case where $\delta=1/2$. To simplify our notation, we define $J:=J_{1/2}$ and $\WFR:=\WFR_{1/2}$.

\subsection{The Square Root Normal Field Shape Distance}
In mathematical shape analysis, one defines metrics that measure the differences between geometric objects~\cite{younes2010shapes,bauer2014overview,srivastava2016functional,dryden2016statistical}. In this article we consider geometric objects described by unparameterized surfaces which are elements of an infinite dimensional non-linear space modulo several finite and infinite dimensional group action. As a result, computations in this space are difficult and even simple statistical operations are not well defined.  Riemannian geometry can help to overcome these challenges. In such a framework, one considers the space of all surfaces as an infinite dimensional manifold and equips it with a Riemannian metric that is invariant to the group action, which allows one to consider the induced metric on the quotient space. 

For our purposes we will consider immersions of a smooth, connected, compact, oriented Riemannian 2-dimensional manifold, $M$, with or without boundary. \textcolor{black}{We will denote the Lesbegue measure on $M$ by $m$.} We denote the space of all Lipschitz immersions of $M$ into $\reals^3$ by $\Imm(M,\reals^3)$, i.e.,
\begin{align}
\Imm(M,\reals^3)=\{f\in W^{1,\infty}(M,\reals^3): Tf \text{ is inj. a.e.}\}\;.    
\end{align}
As we are interested in unparametrized surfaces, we have to factor out the action of the group of diffeomorphisms. In the context of Lipschitz immersions the natural group of reparametrizations for us to consider is the group of all 
orientation preserving, bi-Lipschitz diffeomorphisms:
 \begin{equation*}
     \DiffLip(M)=\{ \gamma\in W^{1,\infty}(M,M):\; \gamma^{-1}\in W^{1,\infty}(M,M), |D\gamma|>0 \text{ a.e.}\},
     \end{equation*} 
where  $|D\gamma|$ denotes the Jacobian determinant of $\gamma$, which is well-defined as $D\gamma\in L^{\infty}$.
Note that  this reparametrization group acts by composition from the right on $\Imm(M,\reals^3)$. In addition to the action by the reparametrization group, we also want to identify surfaces that only differ by a translation. This leads us to consider the following quotient space:
\begin{align}
&\mathcal S:=\Imm(M,\mathbb R^3)/(\DiffLip(M)\times\operatorname{trans}) 
\end{align}
In the following we will equip $\Imm(M)$ with a reparameterization invariant distance; the so called square root normal field (SRNF) distance. The SRNF map (distance resp.) was originally introduced by Jermyn et al. 
in ~\cite{jermyn2012elastic} for the space of smooth immersions, but it naturally  extends to the space of all Lipschitz surfaces, as demonstrated in~\cite{bauer2022SRNF}. We now recall the definition of this distance.

For any given $f\in\Imm(M,\reals^3)$, the orientation on $M$ allows us to consider the unit normal vector field $n_f:M\to\reals^3$, which is well-defined as an element of $L^{\infty}(M,\reals^3)$. Furthermore, let $\{v,w\}$ be an orthonormal basis of $T_xM$. Then for any $f\in\Imm(M,\reals^3)$ we can  define the area multiplication factor at $x\in M$ via $a_f(x)=|df_x(v)\times df_x(w)|$. The SRNF map is then given by  
\begin{align}
\Phi:\Imm(M,\reals^3)/\operatorname{translations}&\to L^2(M,\reals^3)\\
f&\mapsto q_f \text{ where } q_f(x):=\sqrt{a_f(x)}\;n_f(x).
\end{align}
From this transform we define a distance on $\Imm(M,\reals^3)/\operatorname{translations}$ by \begin{equation*}\label{distfuncdef}
d_{\Imm}(f_1,f_2)=\|\Phi(f_1)-\Phi(f_2)\|_{L^2}.
\end{equation*}
Next we consider a right-action of $\DiffLip(M)$ on $L^2(M,\reals^3)$ 
that is compatible with the mapping $\Phi$. For $q\in L^2(M,\reals^3)$ and $\gamma \in \DiffLip(M)$  we let
\begin{align}
 (q*\gamma)(x)=\sqrt{|D_\gamma(x)|}q(\gamma(x)).
\end{align}
It is easy to check that  the action of $\DiffLip(M)$ on $L^2(M,\reals^3)$ is by linear isometries and that for any $f\in\Imm$ and $\gamma\in\DiffLip$,\begin{equation*}
\Phi(f)*\gamma=\Phi(f\circ\gamma).
\end{equation*}Thus, it follows that the SRNF distance on $\Imm(M,\reals^3)$ is invariant with respect to this action and thus it descends to a (pseudo) distance on the quotient space $\mathcal S$, which is given by
\begin{equation*}
   d_{\mathcal S}([f_1],[f_2])=\inf\limits_{\gamma\in\DiffLip(M)}d(f_1,f_2\circ\gamma),\qquad [f_1],[f_2]\in \mathcal{S}(M)
\end{equation*}
As we will see later the induced (pseudo) distance on the quotient space is highly degenerate.
\subsection{Equivalence of WFR and SRNF in the piecewise linear category}
In~\cite{bauer2022SRNF} an equivalence of the WFR and SRNF distance was shown: for piecewise linear surfaces it was proved that the SRNF distance can be reduced to the WFR distance between finitely supported measures. To formulate this result in detail we 
first associate to every $q\in L^2(M,\mathbb{R}^3)$ a measure on $S^2$; namely, for any open $U\subseteq S^2$,  we define 
\begin{equation*}
    q^*U=\{x\in M|q(x)\neq 0 \text{ and } q(x)/|q(x)| \in U \}
\end{equation*}
and define the map
\begin{align*}\label{eq:psi_def}
     &L^2(M,\mathbb{R}^3)\to\mathcal{M}(S^2)\text{ via }q\mapsto\mu_q\\ 
     &\text{where for } U\subseteq S^2, \mu_q(U)=\int_{q^*U}q(x)\cdot q(x)dm.
\end{align*}
The result proved in \cite{bauer2022SRNF} is then formulated as:
\begin{theorem}\label{thm:maintheorem_old}
Given two piecewise linear surfaces $S_1$ and $S_2$ parameterized by $f$ and $g$, the SRNF shape distance can be computed as an unbalanced transport problem. More precisely, we have  
\begin{equation*}
    d_{\mathcal S}([f],[g])=\inf_{\gamma\in\Gamma(M)}\|q_f-q_g*\gamma\|=\WFR(\mu_{q_f},\mu_{q_g}).
\end{equation*}
where $q_f$ and $q_g$ are the SRNFs of $f$ and $g$ respectively.
\end{theorem} 
\noindent In the next section we will  extend this result of  to all Lipschitz immersions (Borel-measures, resp.).

\section{The SRNF distance}
For the goal of extending the result of Theorem~\ref{thm:maintheorem_old} to all Lipschitz surfaces, we will consider specifically $\delta=\frac{1}{2}$ in the definition of the WFR metric. 
\subsection{Equivalence of the WFR and SRNF distances}
Our main result of this section is the following theorem, which is slightly stronger than the desired equivalence result. 
\begin{theorem}\label{thm:maintheorem}
Given $q_1,q_2\in L^2(M,\R^3)$,
\[\inf_{\gamma\in \Gamma(M)}\|q_1-q_2*\gamma\|_{L^2}=\WFR(\mu_{q_1},\mu_{q_2}).\]
In particular, given 
$f,g\in W^{1,\infty}(M,\R^3)$ we can calculate their SRNF distance as an unbalanced OMT problem via
\begin{equation*}
    d_{\mathcal S}([f],[g])=\WFR(\mu_{q_f},\mu_{q_g}),
\end{equation*}
where $q_f$ and $q_g$ are the SRNFs of $f$ and $g$ respectively.
\end{theorem}
\begin{remark}
Note, that as a direct consequence of Theorem \ref{thm:maintheorem} we can also conclude the extension of Theorem \ref{thm:maintheorem_old} to the original setup of the SRNF distance, the space of all smooth surfaces. 
\end{remark}

The proof of Theorem~\ref{thm:maintheorem} relies on a series of technical lemmas, which we will show next.
\begin{lemma}\label{lem:pf-RN}
Let $X,Y$ be topological spaces and $\rho: X\to Y$ be a measurable function with respect to the Borel $\sigma$-algebras. If $\mu,\mu_1\in\mathcal{M}(X)$, $\gamma,\gamma_1\in \mathcal{M}(Y)$ such that $\mu_1\ll\mu$, $\gamma=\rho_*\mu$, and $\gamma_1=\rho_*\mu_1$, then $\gamma_1\ll\gamma$. Furthermore,
$\frac{\mu_1}{\mu}=\frac{\gamma_1}{\gamma}\circ\rho$  almost everywhere.
\end{lemma}
\begin{proof}
Let $U\subseteq Y$ open such that $\gamma(U)=0$. By definition, $\mu(\rho^{-1}(U))=0$. Since $\mu_1\ll\mu$, $\mu_1(\rho^{-1}(U))=0$. Therefore, $\gamma_1(U)=0$. This proves $\gamma_1\ll\gamma$.
\bigskip\\
Following the definitions of the Radon-Nikodym derivatives, pushforwards, and the change of variables formula, we obtain
\begin{align*}
    \int_{\rho^{-1}(U)}\frac{\mu_1}{\mu}d\mu=\int_{\rho^{-1}(U)}d\mu_1=\int_{U}d\gamma_1=\int_{U}\frac{\gamma_1}{\gamma}d\gamma=\int_{\rho^{-1}(U)}\frac{\gamma_1}{\gamma}\circ\rho \,d\mu.
\end{align*}
Thus, $\frac{\mu_1}{\mu}=\frac{\gamma_1}{\gamma}\circ\rho$  almost everywhere.
\end{proof}
Given  $q\in L^2(M,\R^3)$ we can define a function from $M$ to $S^2$ that takes every point $x\in M$ to the unit vector in the direction of $q(x)$. As a matter of defining this function on every point, we can canonically choose the north pole of $S^2$ for points where $q(x)=0$.
\begin{definition}
For $q\in L^2(M,\R^3)$ we define the unit vector map of $q$ as
\begin{align*}
\overline{q}:M&\to S^2 \text{ given by }\\
x&\mapsto \begin{cases}\frac{q(x)}{|q(x)|}& \text{ if } q(x)\neq 0\\(1,0,0)& \text{otherwise}\end{cases}.
\end{align*}
\end{definition}
Note that since $q\in L^2(M,\R^3)$, it follows that  $\overline{q}:M\to S^2$ is measurable.
Let $q\in L^2(M,\R^3)$. We can define a measure, $\nu_q\in \mathcal{M}(M)$, via  \[\nu_q(U)=\int_U |q(x)|^2 dm.\]
for all open $U\subseteq M$.
Note that $\nu_q\ll m$ and $\frac{\nu_q}{m}=|q|^2$. Further, we can equivalently define $\mu_q$ as the pushforward of $\nu_q$ via $\overline{q}$.
\begin{lemma}
Let $q\in L^2(M,\R^3)$ and $\mu_q\in \mathcal{M}(S^2)$ be the measure associated with $q$. Then $\mu_q=\overline{q}_*\nu_q$.
\end{lemma}
\begin{proof}
Let $U\subseteq S^2$ open and define $M_0=\{x\in M|q(x)=0\}$. 

\noindent If $(1,0,0)\not\in S^2$, $\overline{q}^{-1}(U)=q^*(U)$ and thus
\[\overline{q}_*\nu_q(U)=\int_{\overline{q}^{-1}(U)} |q(x)|^2 dm=\int_{q^*(U)}|q(x)|^2dm=\mu_q.\]
If $(1,0,0)\in S^2$, $\overline{q}^{-1}(U)=q^*(U)\cup M_0$ and thus
\[\overline{q}_*\nu_q(U)=\int_{\overline{q}^{-1}(U)} |q(x)|^2 dm =\int_{q^*(U)}|q(x)|^2dm+\int_{M_0}|q(x)|^2 dm =\mu_q.\]
\end{proof}
Leveraging what we have proven above we may show a key continuity result that will then allow us to complete the proof of the main theorem.
\begin{lemma}\label{lem:continuity}
The map $(L^2(M,\R^3),\|\cdot\|_{L^2})\to (\mathcal{M}(S^2),\WFR)$ defined via $q\mapsto \mu_q$ given by Equation~\eqref{eq:psi_def} is Lipschitz continuous with Lipschitz constant $K=1$.
\end{lemma}
\begin{proof}
     Let $q_1,q_2\in L^2(M,\R^3)$. For any semi-coupling $(\gamma_1,\gamma_2)\in\Gamma(\mu_{q_1},\mu_{q_2})$,\[\WFR(\mu_{q_1},\mu_{q_2})\leq \sqrt{J_\delta(\gamma_1,\gamma_2)}.\]
Thus, to prove the theorem we must construct $(\gamma_1,\gamma_2)\in \Gamma(\mu_{q_1},\mu_{q_2})$ such that $J_\delta(\gamma_1,\gamma_2)=\|q_1-q_2\|^2_{L_2}$. To construct such a semi-coupling  we first construct $\rho:M\to S^2\times S^2$ defined as unit vector maps of ${q_1}$ and ${q_2}$ on the first and second factor respectively. I.e. the map is given by $\rho(x)=\left(\overline{q_1}(x),\overline{q_2}(x)\right).$
Since $\overline{q_1}$ and $\overline{q_2}$ are individually measurable, then so is $\rho$. We can then define $\gamma_1,\gamma_2\in\mathcal{M}(S^2\times S^2)$ via $\gamma_1=\rho_*\nu_{q_1}$ and $\gamma_2=\rho_*\nu_{q_2}$. 
\begin{claim}
The pair of measures, $(\gamma_1,\gamma_2)$ is a semi-coupling from $\mu_{q_1}$ to $\mu_{q_2}$.
\end{claim}
\textit{Proof of claim. } Let $U\subseteq S^2$ be open. Thus,
\[\gamma_1(U\times S^2)= \nu_{q_1}\left(\rho^{-1}(U\times S^2)\right)= \nu_{q_1}\left(\overline{{q_1}}^{-1}(U)\cap \overline{q_2}^{-1}(S^2)\right)=\nu_{q_1}\left(\overline{{q_1}}^{-1}(U)\right)=\mu_{q_1}(U) \]and\[\gamma_2(S^2\times U)= \nu_{q_2}\left(\rho^{-1}(S^2\times U)\right)= \nu_{q_1}\left(\overline{{q_1}}^{-1}(S^2)\cap \overline{q_2}^{-1}(U)\right)=\nu_{q_1}\left(\overline{q_2}^{-1}(U)\right)=\mu_{q_2}(U).\]
So $(\gamma_1,\gamma_2)$ is a semi-coupling from $\mu_{q_1}$ to $\mu_{q_2}$.
\bigskip\\
Recall from the definition of the functional $J_\delta$ we need to construct $\gamma\in\mathcal{M}(S^2\times S^2)$ such that $\gamma_1,\gamma_2\ll\gamma$. Define $\gamma= \rho_*m$. We know $\mu_{q_1},\mu_{q_2}\ll m$. Thus, by Lemma \ref{lem:pf-RN}, $\gamma_1,\gamma_2\ll\gamma$. Furthermore,\[|{q_1}|^2=\frac{\mu_{q_1}}{m}=\frac{\gamma_1}{\gamma}\circ \rho\text{ a.e.}\qquad\text{ and }\qquad|q_2|^2=\frac{\mu_{q_2}}{m}=\frac{\gamma_2}{\gamma}\circ \rho \text{ a.e.}\]
So,
\begin{align*}
    J_\delta(\gamma_1,\gamma_2)=&\int_{S^2\times S^2}\left|\sqrt{\frac{\gamma_1}{\gamma}(u,v)}u-\sqrt{\frac{\gamma_1}{\gamma}(u,v)}v\right|^2 d\gamma(u,v)\\
    =&\int_{S^2\times S^2}\frac{\gamma_1}{\gamma}(u,v)d\gamma(u,v)+\int_{S^2\times S^2}\frac{\gamma_2}{\gamma}(u,v)d\gamma(u,v)\\
    &\qquad\qquad\qquad\qquad\qquad\qquad-2\int_{S^2\times S^2}\frac{\sqrt{\gamma_1\gamma_2}}{\gamma}(u,v) \langle u,v\rangle d\gamma(u,v)\\
    =&\int_{\rho^{-1}(S^2\times S^2)}\frac{\gamma_1}{\gamma}\circ \rho (x)\, dm +\int_{\rho^{-1}(S^2\times S^2)}\frac{\gamma_2}{\gamma}\circ \rho (x)\,dm\\
    &\qquad\qquad\qquad\qquad\qquad\qquad-2\int_{\rho^{-1}(S^2\times S^2)}\sqrt{\frac{\gamma_1}{\gamma}\circ\rho(x)}\sqrt{\frac{\gamma_2}{\gamma}\circ\rho(x)}\langle\rho(x)\rangle d\gamma(u,v)\\
    =&\int_{M}|{q_1}(x)|^2 dm+\int_{M} |q_2(x)|^2 dm-2\int_{M} |{q_1}(x)||q_2(x)| \left\langle \frac{{q_1}(x)}{|{q_1}(x)|},\frac{q_2(x)}{|q_2(x)|}\right\rangle dm\\
    =&\|{q_1}-q_2\|^2_{L^2}
\end{align*}
Thus, \[\WFR(\mu_{q_1},\mu_{q_2})\leq \sqrt{J_\delta(\gamma_1,\gamma_2)}= 1 \cdot \|q_1-q_2\|_{L^2}\]
\end{proof}
We are now ready to conclude the proof of Theorem \ref{thm:maintheorem}:
\begin{proof}[Proof of Theorem \ref{thm:maintheorem}]

Let $q_1,q_2\in L^2(M,\R^3)$ and let $\epsilon>0$. Let $p_1,p_2$ be piecewise constant functions such that $\|q_1-p_1\|_{L^2}<\epsilon/4$ and $\|q_2-p_2\|_{L^2}<\epsilon/4$.  Therefore,   \[\inf_{\gamma\in \Gamma(M)}\|q_1-p_1*\gamma\|_{L^2},\,\inf_{\gamma\in \Gamma(M)}\|q_2-p_2*\gamma\|_{L^2},\,  \WFR(\mu_{q_1},\mu_{p_1}),\,\WFR(\mu_{q_2},\mu_{p_2})<\epsilon/4.\]
Thus,
\begin{align*}
    \inf_{\gamma\in \Gamma(M)}\|q_1-q_2*\gamma\|_{L^2}&\leq \inf_{\gamma\in \Gamma(M)}\|q_1-p_1*\gamma\|_{L^2}+\inf_{\gamma\in \Gamma(M)}\|p_2-q_2*\gamma\|_{L^2}\\
    &\qquad\qquad+\inf_{\gamma\in\Gamma(M)}\|p_1-p_2*\gamma\|_{L^2}\\
    &\leq\epsilon/2+\inf_{\gamma\in\Gamma(M)}\|p_1-p_2*\gamma\|_{L^2}\\
    &=\epsilon/2+\WFR(\mu_{p_1},\mu_{p_2})\\
    &\leq \epsilon/2 + \WFR(\mu_{q_1},\mu_{p_1}) + \WFR(\mu_{p_2},\mu_{q_2})+ \WFR(\mu_{q_1},\mu_{q_2})\\
    &\leq \epsilon + \WFR(\mu_{q_1},\mu_{q_2})\\     
    \text{and}\qquad\qquad\qquad\qquad\qquad\\
\end{align*}
\begin{align*}    \WFR(\mu_{q_1},\mu_{q_2})
    &\leq\WFR(\mu_{p_1},\mu_{p_2})+\WFR(\mu_{q_1},\mu_{p_1})+\WFR(\mu_{p_2},\mu_{q_2})\\
    &\leq\WFR(\mu_{p_2},\mu_{q_2})+\epsilon/2\\
    &=\inf_{\gamma\in\Gamma(M)}\|p_1-p_2*\gamma\|_{L^2}+\epsilon/2\\
    &\leq  \inf_{\gamma\in \Gamma(M)}\|q_1-p_1*\gamma\|_{L^2}+\inf_{\gamma\in \Gamma(M)}\|p_2-q_2*\gamma\|_{L^2}\\
    &\qquad\qquad+\inf_{\gamma\in\Gamma(M)}\|q_1-q_2*\gamma\|_{L^2}+\epsilon/2\\
    &\leq \inf_{\gamma\in\Gamma(M)}\|q_1-q_2*\gamma\|_{L^2}+\epsilon.
\end{align*}
So,\begin{equation*}
    \WFR(\mu_{q_1},\mu_{q_2})-\epsilon\leq \inf_{\gamma\in\Gamma(M)}\|q_1-q_2*\gamma\|_{L^2} \leq \WFR(\mu_{q_1},\mu_{q_2})+\epsilon.
\end{equation*}
Taking $\epsilon\to 0$ we can conclude $\inf_{\gamma\in\Gamma(M)}\|q_1-q_2*\gamma\|_{L^2}=\WFR(\mu_{q_1},\mu_{q_2})$.
\end{proof}
\subsection{Characterizing the closure of the image of the SRNF map}
Our equivalence result will also allow us to characterize the (closure of the) image of the SRNF map $\Phi$ in the context of spherical surfaces:
\begin{theorem}\label{thm:image}
Let $f\in\Imm(S^2,\mathbb R^3)$ and let $q=\Phi(f)\in L^2(S^2,\R^3)$.
Then $q$ satisfies the closure condition $\int_{S^2}q(x)|q(x)|\,dm=0$.
Moreover, the closure of the image of $\Phi$ is  given by the set
    \[\mathcal{U}:=\left\{ q\in L^2(S^2,\R^3) \text{ such that } \int_{S^2}q(x)|q(x)|\,dm=0\right\}.\]
\end{theorem}
To prove this result we will need a classical theorem from geometric measure theory and the study of convex polyhedra, which we will recall next:
\begin{theorem}[Minkowski's Theorem \cite{alexandrov1938theorie,Minkowski1897,schneider1993convex}]\label{mink}
Let $\mu\in \mathcal{M}(S^2)$ such that the support of $\mu$ is not concentrated on a great circle and \begin{equation*}
    \int_{S^2}x\,d\mu(x)=0.
\end{equation*}
Then, there exists a unique (up to translation) convex body whose surface area measure is $\mu$. Moreover, if $\mu$ is finitely supported then the convex body is a polytope.
\end{theorem}
\begin{proof}[Proof of Theorem~\ref{thm:image}.]
    Let $f\in \Imm(S^2,\R^3)$ and $q_f=\Phi(f)$. Let $S=f(S^2)$ and $V$ be the surface enclosed by $S$. Therefore, \[\int_{S^2}q_f(x)|q_f(x)|\,dm=\int_{S^2}a_f(x)n_f(x)dm=\int_S n_f dS.\] Thus, this is the integral of the normal vector of a closed surface in $\R^3$. A simple application of the divergence theorem shows that the integral of  the normal vector of the closed surface is zero. To see this, let $\{e_i\}_{i=1}^3$ be the unit basis vectors of $\R^3$. For $i=1,2,3$, \[\int_S (n_f\cdot e_i)\, dS= \int_V (\nabla \cdot e_i) \,dV = 0. \] Therefore, $\int_{S^2}q_f(x)|q_f(x)|\,dm=0$ and the image of $\Phi$ is contained in $\mathcal{U}$. 
    
    To prove the converse direction let $q\in \mathcal{U}$. We aim to construct a convex body $f$ with $\mu_{q_f}$ arbitrarily close to $\mu_q$. By the definition of $\mathcal{U}$ the measure $\mu_q$  satisfies $\int_{S^2}n\,d\mu_{q}(n)=0$. Since finitely supported measures are dense with respect to the $\operatorname{WFR}$ metric, we can choose a finitely supported measure $\overline{\mu_q}$  such that $\int_{S^2}n\,d\overline{\mu_{q}}(n)=0$ and $\operatorname{WFR}(\mu_q,\overline{\mu_q})<\epsilon/3$.
    
    If the support of $\overline{\mu_q}$ is not concentrated on a great circle we can invoke the Minkowski theorem and the result follows. For the general case we will slightly deform the measure as follows.
    Define \[\hat{\mu_q}:=\overline{\mu_q}+\sum_{i=1}^3\frac{\epsilon}{18} \delta_{e_i}+\sum_{i=1}^3\frac{\epsilon}{18} \delta_{-e_i}\] where $\{e_i\}_{i=1}^3$ is the set of unit basis vectors of $\R^3$. Then $\hat{\mu_q}$ is a finitely supported measue and satisfies $\int_{S^2}n\,d\hat{\mu_q}(n)=0$ and $\hat{\mu_q}$ is not supported on a single great circle. Moreover, $\operatorname{WFR}(\overline{\mu_q},\hat{\mu_q})<\epsilon/3$. By the Minkowski Theorem  (Theorem~\ref{mink}) there exists a convex polytope with surface area measure given by $\hat{\mu_{q}}$. Let $f\in W^{1,\infty}(S^2,\R^3)$ be the PL spherical parameterization of this convex body, so that $\mu_{q_f}=\hat{\mu_{q}}$. Thus, there exists $\gamma\in\Gamma(M)$ such that $\|q_f-q*\gamma\|_{L^2}<\operatorname{WFR}(\mu_{q_f},\mu_q)+\epsilon/3$. Therefore, 
    \begin{equation*}
        \|q_f-q*\gamma\|_{L^2}\leq \operatorname{WFR}(\mu_{q_f},\mu_q)+\epsilon/3= \operatorname{WFR}(\hat{\mu_q},\mu_q)+\epsilon/3 \leq\operatorname{WFR}(\hat{\mu_q},\overline{\mu_q})+\operatorname{WFR}(\overline{\mu_q},\mu_q)+\epsilon/3<\epsilon,
    \end{equation*}
    which concludes the proof.
\end{proof}

\subsection{Characterizing the degeneracy of the SRNF distance}
As a second important consequence of the our equivalence result we can give a detailed proof of the degeneracy of the SRNF distance for smooth surfaces. Degeneracy results were studied in \cite{klassen2020closed} and it was further characterized for certain PL  surfaces in \cite{bauer2022SRNF}. Here we will generalize the characterization of \cite{bauer2022SRNF} to smooth surfaces:
\begin{theorem}\label{cor:degeneracy}
For any smooth, regular surface $f\in C^\infty(S^2,\R^3)\cap \Imm(S^2,\mathbb R^3)$ there exists a unique (up to translations) convex body \textcolor{black}{whose boundary (parameterized via the map $f_1$)} is indistinguishable from $f$ by the SRNF shape distance, i.e, $ d_{\mathcal S}([f],[f_1])=0$.
\end{theorem}
\begin{proof}[Proof of Theorem~\ref{cor:degeneracy}]
    Let $f\in C^\infty(S^2,\R^3)\cap \Imm(S^2,\mathbb R^3)$ be a regular surface. By \cite[Prop. 4.33]{tapp2016differential} the Gauss map of $f$ is surjective.
   Thus the image of $q_f$ is not contained in a single hyperplane of $\R^3$. Furthermore, $\int_{S^2}q_f(x)|q_f(x)|\,dm=0$. Thus, by Theorem \ref{mink}, there exists a unique convex body (up to translation) with surface area measure given by $\mu_{q_f}$. By Theorem \ref{thm:maintheorem} the surface $f$ and the \textcolor{black}{boundary of the} convex body are SRNF distance 0 from each other. 
\end{proof}

\section{The WFR metric as a diffeomorphic optimization problem}
In this section, we will generalize the results of the previous sections for the Wasserstein Fisher Rao distance on any manifold and for any coeffecient $\delta$. Thus characterizing the Wasserstein Fisher Rao distance as a diffeomorphic optimization problem. Let $N$ be a smooth, connected, compact, oriented Riemannian manifold. Define the cone over $N$ via
$\mathcal{C}(N):= (N\times \R^{\geq0})/(N\times\{0\})$. If we let $d$ denote the geodesic distance on $N$ and fix some $\delta \in (0,\infty)$, then we can define a metric on $\mathcal{C}(N)$ via
\[d_{\mathcal{C}(N)}((n_1,r_1),(n_2,r_2))^2=4\delta^2r_1^2+4\delta^2r_2^2-8\delta^2r_1r_2 \overline{\cos}(d(n_1,n_2)/2\delta).\]
Let $M$ be another smooth, connected, compact, oriented Riemannian manifold. \textcolor{black}{Again we denote the Lebesgue measure on $M$ by $m$.} Any function $q:M\to \mathcal{C}(N)$  can be decomposed into component functions by $q(x)=(\overline{q}(x),q^\circ(x))$ where $\overline{q}:M\to N$ and $q^\circ:M\to \R^{\geq0}$. We can thus define \[\hat{q}:M\to\R^{\geq0}\text{ via  for all }x\in M,\,\,\hat{q}(x)=\sqrt{2}\delta q^\circ(x).\]
Given $q_1,q_2:M\to\mathcal{C}(N)$. The $L^2$ distance between $q_1$ and $q_2$ is given by \[d_{L^2}(q_1,q_2)^2=\int_M d_{\mathcal{C}(N)}(q_1(x),q_2(x))^2 dm.\]
By decomposing $q_1$ and $q_2$, we can alternatively write 
\begin{equation}
    d_{L^2}(q_1,q_2)^2=\int_M \hat{q_1}(x)^2 dm +\int_M \hat{q_2}(x)^2 dm -2\int_M \hat{q_1}(x)\hat{q_2}(x)\overline{\cos}(d(\overline{q_1}(x),\overline{q_2}(x))/2\delta)dm
\end{equation}
The $L^2$ cost of a function $q:M\to \mathcal{C}(N)$ as the distance from $q$ to the function that maps all of $M$ to the cone point. In particular, using the decomposition of $q$, this distance is given by \[d_{L^2}(0,q)^2 = \int_M \hat{q}(x)^2 \,dm.\] Thus, the space of $L^2$-functions from $M$ to $\mathcal{C}(N)$ as \[L^2(M,\mathcal{C}(N)):=\{q:M\to \mathcal{C}(N) \text{ s.t. } d_{L^2}(0,q)^2<\infty\}\] 
and we equip $L^2(M,\mathcal{C}(N))$ with the metric $d_{L^2}$. We define the right action of the diffeomorphisms of on $L^2(M,\mathcal{C}(N))$ component-wise. We treat $\hat{q}$ as a half density and define the action of $\Gamma(M)$ on this component as the action on half-densities. Thus, we define the action of $\Gamma(M)$ on $L^2(M,\mathcal{C}(N))$ given by
\begin{align*}
    L^2(M,\mathcal{C}(N))\times\Gamma(M)&\to L^2(M,\mathcal{C}(N)) \text{ via }\\
    (\overline{q},\hat{q}),\gamma&\mapsto \left(\overline{q}\circ \gamma,\hat{q}\circ\gamma \cdot \sqrt{|D\gamma|}\right).
\end{align*}
\textcolor{black}{Next we define several measures associated with a function $q\in L^2(M,\mathcal{C}(N))$.} First, we define $\nu_q\in\mathcal{M}(M)$ such that for any $U\subseteq M$ open \[\nu_q(U)=\int_{U}\hat{q}(x)^2dm.\] 
Note that $\nu_q\ll m$ and $\frac{\nu_q}{m}=\hat{q}^2$. Further, we can define a pushforward of $\nu_q$ via $\overline{q}$. In particular, for every $q\in L^2(M,\mathcal{C}(N))$, we can define a Borel measure on $N$ given by $\mu_q:=\overline{q}_*\nu_q.$ In other words for all $U\subseteq N$ open \[\mu_q(U)=\int_{\overline{q}^{-1}(U)} \hat{q}^2(x)dm.\]
The main result of this section is to show that the Wasserstein Fisher Rao distance can be written as the distance between the orbits associated with the measures:
\begin{theorem}\label{thm:cone_iso}
Let $N$ be a smooth connected compact Riemannian manifold and $M$ be a smooth connected compact Riemannian manifold of dimension 2 or higher.
\begin{enumerate}[a.)]
\item For all $\mu_1,\mu_2\in\mathcal{M}(N)$ and $q_1,q_2\in L^2(M,\mathcal{C}(N))$ such that $\mu_1=\overline{q_1}_*\nu_{q_1}$ and $\mu_2=\overline{q_2}_*\nu_{q_2}$ we have
\begin{align*}
    \WFR_\delta(\mu_1,\mu_2)= \inf\limits_{\gamma\in\Gamma(N)}d_{L^2}(q_1,q_2*\gamma).
\end{align*}
\item Moreover, for all $\mu \in \mathcal{M}(N)$ there exists $q\in L^2(M,\mathcal{C}(N))$ such that $\mu=\overline{q}_*\nu_q$. If $\mu$ is a finitely supported measure given by $\mu=\sum_{i=1}^{t_0}a_i\delta_{u_i}$, then one can choose $q$ piecewise constant. More specifically, the function $q$ given by
\[q(x)=\begin{cases}\left(u_j,\sqrt{\frac{a_j}{\text{area}(\sigma_j)}}\right)&\text{ if } 1\leq j\leq t_0\\
(u_1,0)& \text{ if } t_0< j\leq t_1\end{cases},\]
where $\{\sigma_j\}_{j=1}^{t_1}$ is a subdivision of the canonical triangulation of $M$ with $t_1\geq t_0$, satisfies $\mu=\overline{q}_*\nu_q$. 
\end{enumerate}
\end{theorem}
\textcolor{black}{\begin{remark}\label{rem:RelationToCH} In the following remark we will relate the above theorem to the fact that the Wasserstein and Wasserstein Fisher Rao metrics both arise from a Riemannian submersion \cite{OttoPic,GALLOUET20184199}.  We start by briefly recalling these constructions: the celebrated result of Otto~\cite{OttoPic}, shows that \begin{align*} \pi: \Diff(N) & \to \operatorname{Dens}_1(N) \text{ given by } \\ \varphi & \mapsto \varphi_*(\rho_0)
\end{align*} is a formal Riemannian submersion of the $L^2$ metric on $\Diff(N)$ to the Wasserstein metric on the space of smooth probability densities, denoted by $\operatorname{Dens}_1(N)$.  More recently Gallou\"{e}t and Vialard~\cite{GALLOUET20184199} showed that there exists a similar submersion construction in the unbalanced setting. Therefore they considered the automorphism group
$\operatorname{Aut}(\mathcal{C}(N))$, which can be viewed as the semi-direct product of $\Diff(N)$ and $C^\infty(N,\R^+)$, i.e., any element in $\operatorname{Aut}(\mathcal{C}(N))$ can be written as the tuple $(\varphi, f)$ where $\varphi\in\Diff(N)$ and $f\in C^\infty(N,\R^+)$. They then showed that
\begin{align*} \pi_0: \operatorname{Aut}(\mathcal{C}(N)) & \to \operatorname{Dens}(N) \text{ given by } \\ (\varphi,f) & \mapsto \varphi_*(f^2\rho_0)
\end{align*}  is a Riemannian submersion of the $L^2$ metric on $\operatorname{Aut}(\mathcal{C}(N))$ to the Wasserstein Fisher Rao metric on the space of smooth densities, denoted by $\operatorname{Dens}(N)$. In the case where $M=N$, Theorem~\ref{thm:cone_iso} may be viewed to extend the result of \cite{GALLOUET20184199} to both lower regularity measures and more general and lower regularity maps in the top space, albeit without obtaining a Riemannian submersion.  It is, however, easy to see that in the smooth category our result would indeed follow from the results of \cite{GALLOUET20184199}. Therefore we consider the maps
\begin{align*}
    \iota: \operatorname{Aut}(\mathcal{C}(N)) &\to L^2(N,\mathcal{C}(N)) \text{ given by }\\
    (\varphi,f)&\mapsto q \text{ where } q(x)=(\varphi(x),f(x))
\end{align*}
and 
\begin{align*}
    \psi:L^2(N,\mathcal{C}(N))&\to \mathcal{M}(N)\text{ given by }\\
    q&\mapsto \mu_q
\end{align*}
and note that the following diagram commutes.
\begin{equation*}
\begin{tikzcd}[column sep=tiny,ampersand replacement=\&]
      \operatorname{Aut}(\mathcal{C}(N))\arrow[two heads]{dd}{\pi_0}\arrow[r,phantom,"=" description]\& \operatorname{Diff}(N)\arrow[r,phantom,"\ltimes" description]\& C^{\infty}(N,\R^+) \arrow[rightarrow]{rr}{\iota}\&\qquad\&L^2(N,\mathcal{C}(N))\arrow[two heads]{dd}{\psi} \\
      \\
      \operatorname{Dens}(N) \arrow[hookrightarrow]{rrrr}{\text{Inc.}}\&\&\&\& \mathcal{M}(N)\\
\end{tikzcd}
\end{equation*}
Therefore, when $\mu_1,\mu_2\in \operatorname{Dens}(N)$ and $q_1, q_2$  in the image of $\iota$, Theorem~\ref{thm:cone_iso} follows from the results of \cite{GALLOUET20184199}. Finally, we want to emphasize that this only holds in the case where $M=N$ and only in the smooth category, whereas Theorem~\ref{thm:cone_iso} requires nothing of the relationship between $M$ and $N$ and holds for the space of $L^2$-maps (Borel measures, resp.). 
\end{remark}}

Before we are able to prove Theorem~\ref{thm:cone_iso}, we will show again several technical lemmas. Therefore we will consider specific measures associated with functions $q\in L^2(M,\mathcal{C}(N))$. First, we will show that the orbit of any $q\in L^2(M,\mathcal{C}(N))$ under the action of $\Gamma(M)$ is mapped to the same measure on $N$.
\begin{lemma}
Let $q\in L^2(M,\mathcal{C}(N))$. Then for all $\gamma\in \Gamma(M),\,\mu_{q}=\mu_{q*\gamma}.$
\end{lemma}
\begin{proof}
Let $U\subseteq N$ open. Then
\begin{align*}
    \mu_{q*\gamma}(U)&=\int_{\gamma^{-1}(\overline{q}^{-1}(U))} (\hat{q}\circ{\gamma}(x)\cdot\sqrt{|D\gamma|})^2dm\\
    &=\int_{\gamma^{-1}(\overline{q}^{-1}(U))} \hat{q}\circ{\gamma}(x)^2\cdot|D\gamma|\,dm=\int_{\overline{q}^{-1}(U)} \hat{q}(x)^2dm=\mu_{q}(U).
\end{align*}
\end{proof}
\noindent Therefore, we can map each orbit of $q\in L^2(M,\mathcal{C}(N))$ under the half density action by $\Gamma(M)$ to a measure on $N$. 
\noindent As in the previous section, we will first show the result for piecewise constant functions and extend by continuity. We prove the piecewise constant case in the following lemma.
\begin{lemma}\label{lem:cone_iso_dense}
Let $d\geq 2$ and $M$ be a smooth, connected, compact, oriented Riemannian $d$-dimensional manifold with or without boundary. Given two piecewise constant functions $q_1,q_2:M\to \mathcal{C}(N)$, 
\begin{equation*}
    \inf_{\gamma \in \Gamma(M)} d_{L^2}(q_2,q_2*\gamma)=\WFR_\delta(\mu_{q_1},\mu_{q_2}).
\end{equation*}
\end{lemma}

\begin{proof}
Let $\{\sigma_i\}_{i=1}^{t_1}$  and  $\{\tau_j\}_{j=1}^{t_0}$ be triangulations of $M$ such that $q_1$ is constant on each $\sigma_i$ and $q_2$ is constant on each $\tau_j$. Let $\hat{q_1}:M\to \mathbb{R}$, $\overline{q_1}:M\to N$ be the decomposition of $q_1$ and $\hat{q_2}:M\to \mathbb{R}$, $\overline{q_2}:M\to M$ be the decomposition of $q_2$. Define a function
$\langle \cdot,\cdot\rangle:N\times N\to\R \text{  given via  } \langle u,v\rangle=\overline{\cos}(d(u,v)/2\delta).$ A brief computation shows
\begin{align*}
\inf_{\gamma \in \Gamma(M)} d^2_{L^2}(q_1,q_2*\gamma)
=\sum_{i=1}^{t_1}a_i+\sum_{j=1}^{t_0}b_j&-2\sup_{\gamma \in \Gamma(M)}\int_M \hat{q_1}(x)\hat{q_2}(\gamma(x))\sqrt{|D\gamma|}\langle\overline{q_1}(x),\overline{q_2}(\gamma(x))\rangle dm.
\end{align*}
Let $\mathcal{A}$ be the set of all discrete semi-couplings from $\mu_{q_1}$ to $\mu_{q_2}$. Recall
\begin{equation*}
    \WFR_\delta(\mu_{q_1},\mu_{q_2})^2=\sum_{i=1}^{t_1}a_i+\sum_{j=1}^{t_0}b_j-2\sup_{(A,B)\in\mathcal{A}}\sum_{i=1}^{t_1}\sum_{j=1}^{t_0}\sqrt{A_{ij}B_{ij}}\langle u_i,v_j\rangle  
\end{equation*}
Therefore, the theorem is equivalent to showing \begin{equation*} \sup_{(A,B)\in\mathcal{A}}\sum_{i=1}^{t_1}\sum_{j=1}^{t_0}\sqrt{A_{ij}B_{ij}}\langle u_i,v_j\rangle=\sup_{\gamma \in \Gamma(S^2)}\int_M\hat{q_1}(x)\hat{q_2}(\gamma(x))\sqrt{|D\gamma|}\langle\overline{q_1}(x),\overline{q_2}(\gamma(x))\rangle dm.    
\end{equation*} 
\begin{claim}\label{ApproxHom}
 Assume that $(A,B)$ is a discrete semi-coupling from $\mu_{q_1}$ to $\mu_{q_2}$. Then for all $\epsilon>0$ there is a PL homeomorphism $\gamma:M\to M$ such that 
\begin{equation*}\left|\int_M\hat{q_1}(x)\hat{q_2}(\gamma(x))\sqrt{|D\gamma|}\langle\overline{q_1}(x),\overline{q_2}(\gamma(x))\rangle dm-\sum_{i,j}\sqrt{A_{ij}B_{ij}}\langle u_i,v_j\rangle\right|<\epsilon.\end{equation*}
\end{claim}

\noindent\textit{Proof of Claim \ref{ApproxHom}. }Let $(A,B)$ be a discrete semi-coupling from $\mu_{q_1}$ to $\mu_{q_2}$ such that for each $1\leq i\leq t_1$ and $1\leq j\leq t_0$, $A_{ij},B_{ij}>0$. We will first prove the claim for this restricted case and extend it to all semi-couplings by continuity. First we choose a real number $r\in(0,1)$. For each $1\leq i\leq t_1$, subdivide $\sigma_i$ into $n$ smaller $d$-simplexes $\sigma_{ij}$ such that $\hat{q_1}^2=A_{ij}/m(\sigma_{ij})$. Similarly, for each $1\leq j\leq t_0$, subdivide $\tau_j$ into $m$ smaller $d$-simplexes $\tau_{ij}$ such that $\hat{q_2}^2=B_{ij}/m(\tau_{ij})$. For each $1\leq i\leq t_1$ and $1\leq j\leq t_0$, choose a smaller $d$-simplex $\tilde\sigma_{ij}$, whose closure is contained in the interior of $\sigma_{ij}$, such that $m(\tilde\sigma_{ij})=rm(\sigma_{ij})$. Similarly, for each $1\leq i\leq t_1$ and $1\leq j\leq t_0$, choose a smaller $d$-simplex $\tilde\tau_{ij}$, whose closure is contained in the interior of $\tau_{ij}$, such that $m(\tilde\tau_{ij})=rm(\tau_{ij})$. We now construct an orientation preserving PL homeomorphism $\gamma_r:M\to M$. First, for each $1\leq i\leq t_1$ and $1\leq j\leq t_0$, define $\gamma_r:\tilde\sigma_{ij}\to\tilde\tau_{ij}$ to be a PL orientation preserving homeomorphism with constant area multiplication factor, $|D_{\gamma_r}|=m(\tau_{ij})/m(\sigma_{ij})$. Note that \[M-\left(\bigcup\limits_{i=1}^{t_1}\bigcup\limits_{j=1}^{t_0}\tilde\sigma_{ij}^{\mathrm{o}}\right)\text{ is homeomorphic to }M-\left(\bigcup\limits_{i=1}^{t_1}\bigcup\limits_{j=1}^{t_0}\tilde\tau_{ij}^{\mathrm{o}}\right).\] Hence, we can extend the homeomorphism $\gamma_r$ defined on the $\tilde\sigma_{ij}$'s to a homeomorphism from $M$ to $M$.  
Note that on each $\tilde{\sigma}_{ij}$, $\hat{q_2}^2(\gamma_r(x))|D\gamma_r|=B_{ij}/m(\sigma_{ij})$. Write $M=M_1\cup M_2$, where $M_1=\bigcup\limits_{i=1}^{t_1}\bigcup\limits_{j=1}^{t_0}\tilde\sigma_{ij}$ and $M_2=\overline{M-M_1}$. A simple computation shows
\begin{multline*}
\int_{M_1}\hat{q_1}(x)\hat{q_2}(\gamma_r(x))\sqrt{|D\gamma_r|}\langle\overline{q_1}(x),\overline{q_2}(\gamma_r(x))\rangle dm\\
=\sum\limits_{i=1}^{t_1}\sum_{j=1}^{t_0}\int_{\tilde\sigma_{ij}}\hat{q_1}(x)\hat{q_2}(\gamma_r(x))\sqrt{|D\gamma_r|}\langle\overline{q_1}(x),\overline{q_2}(\gamma_r(x))\rangle dm\\
=\sum\limits_{i=1}^{t_1}\sum_{j=1}^{t_0}\frac{\sqrt{A_{ij}B_{ij}}}{m(\sigma_{ij})}m(\tilde\sigma_{ij})\langle u_i,v_j\rangle=\sum\limits_{i=1}^{t_1}\sum_{j=1}^{t_0}\sqrt{rA_{ij}}\sqrt{rB_{ij}}\langle u_i,v_j\rangle.    
\end{multline*}
Meanwhile by the Schwarz inequality,
\begin{multline*}
\left|\int_{M_2}\hat{q_1}(x)\hat{q_2}(\gamma_r(x))\sqrt{|D\gamma_r|}\langle\overline{q_1}(x),\overline{q_2}(\gamma_r(x))\rangle dm\right|
\leq\int_{M_2}\hat{q_1}(x)\hat{q_2}(\gamma_r(x))\sqrt{|D\gamma_r|} dm\\
\leq\sqrt{\int_{M_2}\hat{q_1}^2dm}\sqrt{\int_{M_2}\hat{q_2}^2(\gamma_r(x))|D\gamma_r|\,dm}
=\sqrt{(1-r)\int_{M}\hat{q_1}^2dm }\sqrt{(1-r)\int_{M}\hat{q_2}^2dm }.
\end{multline*}
So as we let $r\to 1$,
\begin{equation*}
    \int_{M_1}\hat{q_1}(x)\hat{q_2}(\gamma_r(x))\sqrt{|D\gamma_r|}\langle\overline{q_1}(x),\overline{q_2}(\gamma_r(x))\rangle dm \to\sum\limits_{i=1}^{t_1}\sum_{j=1}^{t_0}\sqrt{A_{ij}B_{ij}}\langle u_i,v_j\rangle
\end{equation*}
and
\begin{equation*}
\int_{M_2}\hat{q_1}(x)\hat{q_2}(\gamma_r(x))\sqrt{|D\gamma_r|}\langle\overline{q_1}(x),\overline{q_2}(\gamma_r(x))\rangle dm \to 0.
\end{equation*}
Hence, 
\begin{equation*}
    \int_{M}\hat{q_1}(x)\hat{q_2}(\gamma_r(x))\sqrt{|D\gamma_r|}\langle\overline{q_1}(x),\overline{q_2}(\gamma_r(x))\rangle dm \to\sum\limits_{i=1}^{t_1}\sum_{j=1}^{t_0}\sqrt{A_{ij}B_{ij}}\langle u_i,v_j\rangle. 
\end{equation*}Thus Claim \ref{ApproxHom} follows for the case in which for each $1\leq i\leq t_1$ and $1\leq j\leq t_0$, $A_{ij}>0$ and $B_{ij}>0$. The general case then follows immediately from the continuity of
\begin{equation*}
    \sum\limits_{i=1}^{t_1}\sum_{j=1}^{t_0}\sqrt{A_{ij}B_{ij}}\langle u_i,v_j\rangle
\end{equation*} as a function of $(A,B)$.
This completes the proof of Claim \ref{ApproxHom}. It follows that \begin{equation*}
  \sup_{\gamma \in \Gamma(S^2)}\int_M \hat{q_1}(x)\hat{q_2}(x)\langle\overline{q_1}(x),\overline{q_2}(x)\rangle dm\geq\sup_{(A,B)\in\mathcal{A}}\sum_{i=1}^{t_1}\sum_{j=1}^{t_0}\sqrt{A_{ij}B_{ij}}\langle u_i,v_j\rangle .
\end{equation*}
We are left to show the opposite inequality.
\begin{claim}\label{maintheoremclaim3}
 Assume $\gamma$ is a PL-homeomorphism from $M$ to $M$, then there exists a discrete semi-coupling $(A,B)$ such that \begin{equation*} \sup_{\gamma \in \Gamma(M)}\int_M\hat{q_1}(x)\hat{q_2}(\gamma(x))\sqrt{|D\gamma|}\langle\overline{q_1}(x),\overline{q_2}(\gamma(x))\rangle dm\leq\sup_{(A,B)\in\mathcal{A}}\sum_{i=1}^{t_1}\sum_{j=1}^{t_0}\sqrt{A_{ij}B_{ij}}\langle u_i,v_j\rangle .\end{equation*}
\end{claim}
\noindent\textit{Proof of Claim \ref{maintheoremclaim3}. }Let $\gamma:M\to M$ be an orientation preserving PL homeomorphism. For $1\leq i\leq t_1$ and $1\leq j\leq t_0$, define $\sigma_{ij}=\gamma^{-1}(\tau_j)\cap\sigma_i$ and define $\tau_{ij}=\gamma(\sigma_{ij})$. Now define two $(m+1)\times (n+1)$ matrices $A$ and $B$ via:
\begin{itemize}
    \item For $1\leq i\leq t_1$ and $1\leq j\leq t_0$,\, ${\displaystyle A_{ij}=\int_{\sigma_{ij}}\hat{q_1}^2dm \text{ and } B_{ij}=\int_{\tau_{ij}}\hat{q_2}^2dm}.$
    \item For $0\leq i\leq t_1$, $B_{0i}=0\text{ and }\displaystyle{A_{i0}=a_i-\sum\limits_{j=1}^{t_0}\int_{\sigma_{ij}}\hat{q_1}^2dm}.$
    \item For $0\leq j\leq t_0$, $A_{j0}=0\text{ and }\displaystyle{B_{0j}=b_j-\sum\limits_{i=1}^{t_1}\int_{\tau_{ij}}\hat{q_2}^2dm}.$
\end{itemize}

\noindent The pair of matrices $(A,B)$ is a discrete semi-coupling from $\mu_{q_1}$ to $\mu_{q_2}$ by construction. We say that $(A,B)$ is the semi-coupling corresponding to the homeomorphism $\gamma$. Denote the area multiplication factor of $\gamma$ on $\sigma_{ij}$ by $m_{ij}$. Then by the Schwarz inequality,
\begin{multline*}
    \int_{\sigma_{ij}} \hat{q_1}(x)\hat{q_2}(\gamma(x))\sqrt{|D\gamma|}\langle u_i,v_j\rangle dm\leq \sqrt{\int_{\sigma_{ij}} \hat{q_1}^2(x) dm}\sqrt{\int_{\sigma_{ij}} \hat{q_2}^2(\gamma(x))|D\gamma|\,dm} \langle  u_i\cdot v_j\rangle\\
    =\sqrt{\int_{\sigma_{ij}} \hat{q_1}^2(x) dm}\sqrt{\int_{\tau_{ij}} \hat{q_2}^2(x)dm} \langle  u_i\cdot v_j\rangle=\sqrt{A_{ij}}\sqrt{B_{ij}}\langle u_i\cdot v_j\rangle.
\end{multline*}
Summing over all $i$ and $j$ we obtain:
\begin{multline*}
    \int_{M} \hat{q_1}(x)\hat{q_2}(\gamma(x))\sqrt{|D\gamma|}\langle\overline{q_1}(x),\overline{q_2}(\gamma(x))\rangle dm \\=\sum_{i,j}\int_{\sigma_{ij}} \hat{q_1}(x)\hat{q_2}(\gamma(x))\sqrt{|D\gamma|}\langle\overline{q_1}(x),\overline{q_2}(\gamma(x))\rangle dm\leq \sum_{i,j}\sqrt{A_{ij}}\sqrt{B_{ij}}\langle u_i\cdot v_j\rangle.
\end{multline*}

This completes the proof of Claim \ref{maintheoremclaim3}. It follows that,
\begin{equation*} \sup_{\gamma \in \Gamma(M)}\int_M \hat{q_1}(x)\hat{q_2}(\gamma(x))\sqrt{|D\gamma|}\langle\overline{q_1}(x),\overline{q_2}(\gamma(x))\rangle dm\leq\sup_{(A,B)\in\mathcal{A}}\sum_{i=1}^{t_1}\sum_{j=1}^{t_0}\sqrt{A_{ij}B_{ij}}\langle u_i\cdot v_j\rangle.\end{equation*}
and thus the lemma is proved.
\end{proof}
To extend the results to all of $L^2(M,\mathcal{C}(N))$ we will need  the following continuity result:
\begin{lemma}\label{lem:continuity_cone}
The map $(L^2(M,\mathcal{C}(N)),d_{L^2})\to (\mathcal{M}(N),   \WFR_\delta)$ defined via $q\mapsto \overline{q}_*\nu_q$ is Lipschitz continuous with Lipschitz constant $K=1$.
\end{lemma}
\begin{proof}Let $q_1,q_2\in L^2(M,\mathcal{C}(N))$, $\mu_{q_1}= \overline{q_1}_*\nu_{q_1}$, and $\mu_{q_2}= \overline{q_2}_*\nu_{q_2}$. For any semi-coupling $(\gamma_1,\gamma_2)\in\Gamma(\mu_{q_1},\mu_{q_2})$,\[\WFR_\delta(\mu_{q_1},\mu_{q_2})\leq \sqrt{J_\delta(\gamma_1,\gamma_2)}.\]
Thus, to prove the theorem we must construct $(\gamma_1,\gamma_2)\in \Gamma(\mu_{q_1},\mu_{q_2})$ such that $J_\delta(\gamma_1,\gamma_2)=d_{L_2}(q_1,q_2)$. To construct such a semi-coupling  we first construct $\rho:M\to N\times N$ defined a the first component maps of ${q_1}$ and ${q_2}$ on the first and second factor respectively. I.e. the map is given by $\rho(x)=\left(\overline{q_1}(x),\overline{q_2}(x)\right).$
Since $\overline{q_1}$ and $\overline{q_2}$ are individually measurable, then so is $\rho$. We can then define $\gamma_1,\gamma_2\in\mathcal{M}(N\times N)$ via $\gamma_1=\rho_*\nu_{q_1}$ and $\gamma_2=\rho_*\nu_{q_2}$. 
\begin{claim}
The pair of measures, $(\gamma_1,\gamma_2)$ is a semi-coupling from $\mu_{q_1}$ to $\mu_{q_2}$.
\end{claim}
\noindent\textit{Proof of claim. } Let $U\subseteq N$ be open. Thus,
\[\gamma_1(U\times N)= \nu_{q_1}\left(\rho^{-1}(U\times N)\right)= \nu_{q_1}\left(\overline{{q_1}}^{-1}(U)\cap \overline{q_2}^{-1}(N)\right)=\nu_{q_1}\left(\overline{{q_1}}^{-1}(U)\right)=\mu_{q_1}(U) \]
and
\[\gamma_2(N\times U)= \nu_{q_2}\left(\rho^{-1}(N\times U)\right)= \nu_{q_1}\left(\overline{{q_1}}^{-1}(N)\cap \overline{q_2}^{-1}(U)\right)=\nu_{q_1}\left(\overline{q_2}^{-1}(U)\right)=\mu_{q_2}(U).\]
So $(\gamma_1,\gamma_2)$ is a semi-coupling from $\mu_{q_1}$ to $\mu_{q_2}$.
\bigskip\\
Recall from the definition of the functional $J$ we need to construct $\gamma\in\mathcal{M}(N\times N  )$ such that $\gamma_1,\gamma_2\ll\gamma$. Define $\gamma= \rho_*m$. We know $\mu_{q_1},\mu_{q_2}\ll m$. Thus, by Lemma \ref{lem:pf-RN}, $\gamma_1,\gamma_2\ll\gamma$. Furthermore,\[\hat{q_1}^2=\frac{\mu_{q_1}}{m}=\frac{\gamma_1}{\gamma}\circ \rho\text{ a.e.}\qquad\text{ and }\qquad\hat{q_2}^2=\frac{\mu_{q_2}}{m}=\frac{\gamma_2}{\gamma}\circ \rho \text{ a.e.}\]
So,
\begin{align*}
    J_\delta(\gamma_1,\gamma_2)=&\mu_1(N)+\mu_2(N)-2\int_{N\times N}\frac{\sqrt{\gamma_1\gamma_2}}{\gamma}(u,v) \overline{\cos}(d(u,v)/2\delta) d\gamma(u,v)\\
    =&\int_{N\times N}\frac{\gamma_1}{\gamma}\,d\gamma+\int_{N\times N}\frac{\gamma_2}{\gamma}\,d\gamma-2\int_{N\times N}\sqrt{\frac{\gamma_1}{\gamma}(u,v)\frac{\gamma_2}{\gamma}(u,v)} \overline{\cos}(d(u,v)/2\delta) d\gamma(u,v)\\
    =&\int_{\rho^{-1}(N\times N)}\frac{\gamma_1}{\gamma}\circ\rho\,dm+\int_{\rho^{-1}(N\times N)}\frac{\gamma_2}{\gamma}\circ\rho\,dm\\
    &\qquad\qquad\qquad\qquad\qquad-2\int_{\rho^{-1}(N\times N)}\sqrt{\frac{\gamma_1}{\gamma}\circ\rho(x)\frac{\gamma_2}{\gamma}\circ\rho(x)} \overline{\cos}(d(\rho(x))/2\delta) dm\\
    =&\int_{M}\hat{q_1}(x)^2\,dm+\int_{M}\hat{q_2}(x)^2\,dm-2\int_{M}\hat{q_1}(x)\hat{q_2}(x) \overline{\cos}(d(\overline{q_1},\overline{q_2})/2\delta) dm = d_{L^2}(q_1,q_2)
\end{align*}
Thus, \[\WFR_\delta(\mu_{q_1},\mu_{q_2})\leq \sqrt{J_\delta(\gamma_1,\gamma_2)}= 1 \cdot d_{L^2}(\mu_{q_1},\mu_{q_2})\]
\end{proof}
\noindent Finally, we can leverage this continuity result to complete the proof of Theorem \ref{thm:cone_iso}.\\
\begin{proof}[Proof of Theorem \ref{thm:cone_iso}] Let $\mu_1,\mu_2\in \mathcal{M}(N)$ and $q_1,q_2\in L^2(M,\mathcal{C}(N))$ such that $\mu_{1}= \overline{q_1}_*\nu_{q_1}$ and $\mu_{2}= \overline{q_2}_*\nu_{q_2}$. By an argument analogous to the proof of Theorem \ref{thm:maintheorem} we can conclude \[\inf_{\gamma\in \Gamma(M)}d_{L^2}(q_1,q_2*\gamma)=\WFR_\delta(\mu_{1},\mu_{2}).\]
This concludes the the proof of part a.). Let $\mu=\sum_{i=1}^{t_0}a_i\delta_{u_i}$ be a finitely supported measure on $N$.  By~\cite{whitehead1940c1}, $M$ admits a canonical PL structure. Let $t_1\geq t_0$ and subdivide the triangulation of $M$ into $t_1$ simplices given by $\sigma_j$ for $1\leq j\leq t_1$. Let $x\in M$. Thus, there exists $1\leq j\leq t_1$ such that $x\in \sigma_j$. Thus we define \[q(x)=\begin{cases}\left(u_j,\sqrt{\frac{a_j}{\text{area}(\sigma_j)}}\right)&\text{ if } 1\leq j\leq t_0\\
(u_1,0)& \text{ if } n< j\leq t_1\end{cases}.\] Let $U\subseteq N$, then $\mu(U)=\sum\limits_{i|u_i\in U}a_i$. Meanwhile, $\overline{q}^{-1}(U)=\bigsqcup\limits_{i|u_i\in U}\sigma_i$. Thus, \[\int_{\overline{q}^{-1}(U)} \hat{q}^2(x)dm=\sum_{i|u_i\in U}\int_{\sigma_i} \frac{a_i}{\text{area}(\sigma_i)}dm=\sum_{i|u_i\in U}a_i.\] To complete the proof of part b.) we will extend the result to the whole space by continuity. For any $\mu\in\mathcal{M}(N)$, let $\{\mu_n\}\subseteq \mathcal{M}(N)$ be a sequence of finitely supported measures that converges to $\mu$ with respect to the Wasserstein Fisher Rao. In particular, $\{\mu_n\}$ is Cauchy with respect to $\WFR_\delta$. Note that for all $n\in\mathbb{N}$,there exists a piecewise constant $q_n\in  L^2(M,\mathcal{C}(N))$ satisfying \[\mu_n(U)=\int_{\overline{q_n}^{-1}(U)} \hat{q_n}(x)^2dm.\]
Thus, we can construct a sequence of functions given by $q^*_0=q_0$ an for all $n\in \mathbb{N}$,  $q*_{n+1}=q_{n+1}*\gamma_n$ where $\gamma_n$ is a PL homeomorphism from $M$ to $M$ such that \[d_{L^2}(q^*_n, q_{n+1}*\gamma_n)=\WFR_\delta(\mu_n,\mu_{n+1})+\frac{1}{2^{t_0}}.\]
Note that the existence of such a $\gamma_n$ is guaranteed by Lemma \ref{lem:cone_iso_dense}. Since  $\{\mu_n\}$ is Cauchy with respect to $\WFR_\delta$, it follows that $\{q^*_n\}$ is Cauchy with respect to $d_{L^2}$. By completeness of  $(L^2(M,\mathcal{C}(N)),d_{L^2})$, there exists a limit $q\in L^2(M,\mathcal{C}(N))$. Let $U\subseteq N$ open. Thus,
\begin{align*}
    \mu(U)=&\lim\limits_{n\to\infty}\mu_n(U)=\lim\limits_{n\to\infty}\int_{\overline{q_n}^{-1}(U)} \hat{q_n}(x)^2dm=\lim\limits_{n\to\infty}\int_M \hat{q_n}(x)^2\chi_{\overline{q_n}^{-1}(U)}dm\\
    =&\int_M\lim\limits_{n\to\infty} \hat{q_n}(x)^2\chi_{\overline{q_n}^{-1}(U)}dm=\int_M\hat{q}(x)^2\chi_{\overline{q}^{-1}(U)}dm=\int_{\overline{q}^{-1}(U)}\hat{q}(x)^2 dm
\end{align*}
Thus, $\mu=\overline{q}_*\nu_{q}$ This completes the proof of part b.) of the theorem.
\end{proof}

\bibliographystyle{siamplain}

\begin{thebibliography}{10}

\bibitem{alexandrov1938theorie}
{\sc A.~Alexandrov}, {\em Zur Theorie der Gemischten Volumina von Konvexen K{\"o}rpern i}, Mat. Sbornik NS, 1 (1938), pp.~227--251.

\bibitem{bauer2017numerical}
{\sc M.~Bauer, M.~Bruveris, P.~Harms, and J.~M{\o}ller-Andersen}, {\em A Numerical Framework for Sobolev metrics on the Space of Curves}, SIAM Journal on Imaging Sciences, 10 (2017), pp.~47--73.

\bibitem{bauer2014overview}
{\sc M.~Bauer, M.~Bruveris, and P.~W. Michor}, {\em Overview of the Geometries of Shape Spaces and Diffeomorphism Groups}, Journal of Mathematical Imaging and Vision, 50 (2014), pp.~60--97.

\bibitem{bauer2020numerical}
{\sc M.~Bauer, N.~Charon, P.~Harms, and H.-W. Hsieh}, {\em A Numerical Framework for Elastic Surface Matching, Comparison, and Interpolation}, International Journal of Computer Vision, 129 (2021), pp.~2425--2444.

\bibitem{bauer2022elastic}
{\sc M.~Bauer, N.~Charon, E.~Klassen, S.~Kurtek, T.~Needham, and T.~Pierron}, {\em Elastic metrics on spaces of euclidean curves: Theory and algorithms}, arXiv preprint arXiv:2209.09862,  (2022).

\bibitem{bauer2022SRNF}
{\sc M.~Bauer, E.~Hartman, and E.~Klassen}, {\em The square root normal field distance and unbalanced optimal transport}, Applied Mathematics {\&} Optimization, 85 (2022), \url{https://doi.org/10.1007/s00245-022-09867-y}, \url{https://doi.org/10.1007\%2Fs00245-022-09867-y}.

\bibitem{bruveris2016optimal}
{\sc M.~Bruveris}, {\em Optimal reparametrizations in the square root velocity framework}, SIAM Journal on Mathematical Analysis, 48 (2016), pp.~4335--4354.

\bibitem{charon2022shape}
{\sc N.~Charon and L.~Younes}, {\em Shape spaces: From geometry to biological plausibility}, arXiv preprint arXiv:2205.01237,  (2022).

\bibitem{chizat2018interpolating}
{\sc L.~Chizat, G.~Peyr{\'e}, B.~Schmitzer, and F.-X. Vialard}, {\em An interpolating distance between optimal transport and {F}isher--{R}ao metrics}, Foundations of Computational Mathematics, 18 (2018), pp.~1--44.

\bibitem{chizat2018scaling}
{\sc L.~Chizat, G.~Peyr{\'e}, B.~Schmitzer, and F.-X. Vialard}, {\em Scaling algorithms for unbalanced optimal transport problems}, Mathematics of Computation, 87 (2018), pp.~2563--2609.

\bibitem{chizat2018unbalanced}
{\sc L.~Chizat, G.~Peyr{\'e}, B.~Schmitzer, and F.-X. Vialard}, {\em Unbalanced optimal transport: Dynamic and {K}antorovich formulations}, Journal of Functional Analysis, 274 (2018), pp.~3090--3123.

\bibitem{dogan2015fast}
{\sc G.~Dogan, J.~Bernal, and C.~R. Hagwood}, {\em A fast algorithm for elastic shape distances between closed planar curves}, in Proceedings of the IEEE Conference on Computer Vision and Pattern Recognition, 2015, pp.~4222--4230.

\bibitem{dryden2016statistical}
{\sc I.~L. Dryden and K.~V. Mardia}, {\em Statistical shape analysis: with applications in R}, vol.~995, John Wiley \& Sons, 2016.

\bibitem{GALLOUET20184199}
{\sc T.~Gallouët and F.-X. Vialard}, {\em The camassa–holm equation as an incompressible euler equation: A geometric point of view}, Journal of Differential Equations, 264 (2018), pp.~4199--4234, \url{https://doi.org/https://doi.org/10.1016/j.jde.2017.12.008}, \url{https://www.sciencedirect.com/science/article/pii/S0022039617306435}.

\bibitem{hartman2022elastic}
{\sc E.~Hartman, Y.~Sukurdeep, E.~Klassen, N.~Charon, and M.~Bauer}, {\em Elastic shape analysis of surfaces with second-order sobolev metrics: a comprehensive numerical framework}, To appear in IJCV,  (2022).

\bibitem{jermyn2012elastic}
{\sc I.~H. Jermyn, S.~Kurtek, E.~Klassen, and A.~Srivastava}, {\em Elastic shape matching of parameterized surfaces using square root normal fields}, in European conference on computer vision, Springer, 2012, pp.~804--817.

\bibitem{jermyn2017elastic}
{\sc I.~H. Jermyn, S.~Kurtek, H.~Laga, and A.~Srivastava}, {\em Elastic shape analysis of three-dimensional objects}, Synthesis Lectures on Computer Vision, 12 (2017), pp.~1--185.

\bibitem{joshi2016surface}
{\sc S.~H. Joshi, Q.~Xie, S.~Kurtek, A.~Srivastava, and H.~Laga}, {\em Surface shape morphometry for hippocampal modeling in alzheimer's disease}, in 2016 International Conference on Digital Image Computing: Techniques and Applications (DICTA), IEEE, 2016, pp.~1--8.

\bibitem{klassen2020closed}
{\sc E.~Klassen and P.~W. Michor}, {\em Closed surfaces with different shapes that are indistinguishable by the {SRNF}}, Archivum Mathematicum, 56 (2020), pp.~107--114.

\bibitem{kondratyev2016new}
{\sc S.~Kondratyev, L.~Monsaingeon, D.~Vorotnikov, et~al.}, {\em A new optimal transport distance on the space of finite radon measures}, Advances in Differential Equations, 21 (2016), pp.~1117--1164.

\bibitem{kurtek2010novel}
{\sc S.~Kurtek, E.~Klassen, Z.~Ding, and A.~Srivastava}, {\em A novel {R}iemannian framework for shape analysis of 3{D} objects}, in 2010 IEEE computer society conference on computer vision and pattern recognition, IEEE, 2010, pp.~1625--1632.

\bibitem{kurtek2011elastic}
{\sc S.~Kurtek, E.~Klassen, J.~C. Gore, Z.~Ding, and A.~Srivastava}, {\em Elastic geodesic paths in shape space of parameterized surfaces}, IEEE transactions on pattern analysis and machine intelligence, 34 (2011), pp.~1717--1730.

\bibitem{kurtek2014statistical}
{\sc S.~Kurtek, C.~Samir, and L.~Ouchchane}, {\em Statistical shape model for simulation of realistic endometrial tissue.}, in ICPRAM, 2014, pp.~421--428.

\bibitem{laga20214d}
{\sc H.~Laga, M.~Padilla, I.~H. Jermyn, S.~Kurtek, M.~Bennamoun, and A.~Srivastava}, {\em 4d atlas: Statistical analysis of the spatiotemporal variability in longitudinal 3{D} shape data}, arXiv preprint arXiv:2101.09403,  (2021).

\bibitem{laga2017numerical}
{\sc H.~Laga, Q.~Xie, I.~H. Jermyn, and A.~Srivastava}, {\em Numerical inversion of {SRNF} maps for elastic shape analysis of genus-zero surfaces}, IEEE transactions on pattern analysis and machine intelligence, 39 (2017), pp.~2451--2464.

\bibitem{lahiri2015precise}
{\sc S.~Lahiri, D.~Robinson, and E.~Klassen}, {\em {Precise matching of PL curves in $\mathbb{R}^N$ in the square root velocity framework}}, Geometry, Imaging and Computing, 2 (2015), pp.~133--186.

\bibitem{liero2016optimal}
{\sc M.~Liero, A.~Mielke, and G.~Savar{\'e}}, {\em Optimal transport in competition with reaction: The {H}ellinger--{K}antorovich distance and geodesic curves}, SIAM Journal on Mathematical Analysis, 48 (2016), pp.~2869--2911.

\bibitem{liero2018optimal}
{\sc M.~Liero, A.~Mielke, and G.~Savar{\'e}}, {\em Optimal entropy-transport problems and a new {H}ellinger--{K}antorovich distance between positive measures}, Inventiones mathematicae, 211 (2018), pp.~969--1117.

\bibitem{Lombardi2015}
{\sc {Lombardi, Damiano} and {Maitre, Emmanuel}}, {\em Eulerian models and algorithms for unbalanced optimal transport}, ESAIM: M2AN, 49 (2015), pp.~1717--1744, \url{https://doi.org/10.1051/m2an/2015025}, \url{https://doi.org/10.1051/m2an/2015025}.

\bibitem{marron2014overview}
{\sc J.~S. Marron and A.~M. Alonso}, {\em Overview of object oriented data analysis}, Biometrical Journal, 56 (2014), pp.~732--753.

\bibitem{matuk2020biomedical}
{\sc J.~Matuk, S.~Mohammed, S.~Kurtek, and K.~Bharath}, {\em Biomedical applications of geometric functional data analysis}, in Handbook of Variational Methods for Nonlinear Geometric Data, Springer, 2020, pp.~675--701.

\bibitem{michor2007overview}
{\sc P.~W. Michor and D.~Mumford}, {\em An overview of the Riemannian metrics on spaces of curves using the Hamiltonian approach}, Applied and Computational Harmonic Analysis, 23 (2007), pp.~74--113.

\bibitem{miller2002metrics}
{\sc M.~I. Miller, A.~Trouv{\'e}, and L.~Younes}, {\em On the Metrics and Euler-Lagrange equations of computational anatomy}, Annual review of biomedical engineering, 4 (2002), pp.~375--405.

\bibitem{Minkowski1897}
{\sc H.~Minkowski}, {\em Allgemeine Lehrsätze über die Convexen Polyeder}, Nachrichten von der Gesellschaft der Wissenschaften zu Göttingen, Mathematisch-Physikalische Klasse, 1897 (1897), pp.~198--220, \url{http://eudml.org/doc/58391}.

\bibitem{mio2007shape}
{\sc W.~Mio, A.~Srivastava, and S.~Joshi}, {\em On Shape of Plane Elastic Curves}, International Journal of Computer Vision, 73 (2007), pp.~307--324.

\bibitem{needham2020simplifying}
{\sc T.~Needham and S.~Kurtek}, {\em Simplifying Transforms for General Elastic Metrics on the Space of Plane Curves}, SIAM journal on imaging sciences, 13 (2020), pp.~445--473.

\bibitem{OttoPic}
{\sc F.~Otto}, {\em The Geometry of Dissipative evolution equations: The porous medium equation}, Communications in Partial Differential Equations, 26 (2001), pp.~101--174, \url{https://doi.org/10.1081/PDE-100002243}, \url{https://doi.org/10.1081/PDE-100002243}, \url{https://arxiv.org/abs/https://doi.org/10.1081/PDE-100002243}.

\bibitem{pennec2006intrinsic}
{\sc X.~Pennec}, {\em Intrinsic Statistics on Riemannian manifolds: Basic tools for Geometric Measurements}, Journal of Mathematical Imaging and Vision, 25 (2006), pp.~127--154.

\bibitem{pennec2019riemannian}
{\sc X.~Pennec, S.~Sommer, and T.~Fletcher}, {\em Riemannian Geometric Statistics in Medical Image Analysis}, Academic Press, 2019.

\bibitem{piccoli2014generalized}
{\sc B.~Piccoli and F.~Rossi}, {\em Generalized {W}asserstein distance and its application to transport equations with source}, Archive for Rational Mechanics and Analysis, 211 (2014), pp.~335--358.

\bibitem{schneider1993convex}
{\sc R.~Schneider}, {\em Convex Surfaces, Curvature and Surface Area Measures}, in Handbook of convex geometry, Elsevier, 1993, pp.~273--299.

\bibitem{srivastava2010shape}
{\sc A.~Srivastava, E.~Klassen, S.~H. Joshi, and I.~H. Jermyn}, {\em Shape Analysis of Elastic Curves in Euclidean Spaces}, IEEE transactions on pattern analysis and machine intelligence, 33 (2010), pp.~1415--1428.

\bibitem{srivastava2016functional}
{\sc A.~Srivastava and E.~P. Klassen}, {\em Functional and Shape Data Analysis}, vol.~1, Springer, 2016.

\bibitem{su2020simplifying}
{\sc Z.~Su, M.~Bauer, E.~Klassen, and K.~Gallivan}, {\em Simplifying Transformations for a family of Elastic Metrics on the Space of Surfaces}, in Proceedings of the IEEE/CVF Conference on Computer Vision and Pattern Recognition Workshops, 2020, pp.~848--849.

\bibitem{su2020shape}
{\sc Z.~Su, M.~Bauer, S.~C. Preston, H.~Laga, and E.~Klassen}, {\em Shape Analysis of Surfaces using General Elastic Metrics}, Journal of Mathematical Imaging and Vision, 62 (2020), pp.~1087--1106.

\bibitem{tapp2016differential}
{\sc K.~Tapp}, {\em Differential Geometry of Curves and Surfaces}, Undergraduate Texts in Mathematics, Springer International Publishing, 2016, \url{https://books.google.com/books?id=kfIqDQAAQBAJ}.

\bibitem{trouve2000class}
{\sc A.~Trouv{\'e} and L.~Younes}, {\em On a Class of Diffeomorphic Matching Problems in one dimension}, SIAM Journal on Control and Optimization, 39 (2000), pp.~1112--1135.

\bibitem{villani2003topics}
{\sc C.~Villani}, {\em Topics in Optimal Transportation}, no.~58, American Mathematical Soc., 2003.

\bibitem{villani2008optimal}
{\sc C.~Villani}, {\em Optimal Transport: old and new}, vol.~338, Springer Science \& Business Media, 2008.

\bibitem{Wang2020RobustDD}
{\sc Z.~Wang, D.~P. Zhou, M.~Yang, Y.~Zhang, C.-Y. Rao, and H.~Wu}, {\em Robust Document Distance with Wasserstein-Fisher-Rao metric}, in ACML, 2020.

\bibitem{whitehead1940c1}
{\sc J.~H.~C. Whitehead}, {\em {On {C}1-complexes}}, Annals of Mathematics,  (1940), pp.~809--824.

\bibitem{woien2022pde}
{\sc E.~N. W{\o}ien and M.~Grasmair}, {\em A Pde-Based method for Shape Registration}, SIAM Journal on Imaging Sciences, 15 (2022), pp.~762--796.

\bibitem{younes1998computable}
{\sc L.~Younes}, {\em Computable Elastic Distances Between Shapes}, SIAM Journal on Applied Mathematics, 58 (1998), pp.~565--586.

\bibitem{younes2010shapes}
{\sc L.~Younes}, {\em Shapes and Diffeomorphisms}, vol.~171, Springer, 2010.

\bibitem{Zhou2018TheWM}
{\sc D.~Zhou, J.~Chen, H.~Wu, D.~H. Yang, and L.~Qiu}, {\em The Wasserstein-Fisher-Rao Metric for Waveform Based Earthquake Location}, arXiv: Numerical Analysis,  (2018).

\end{thebibliography}

\end{document}